\documentclass[12pt]{amsart} 
\usepackage{amscd}
\usepackage{amssymb}
\usepackage{a4wide}
\usepackage{amstext}
\usepackage{amsthm}
\usepackage{xcolor}
\usepackage[T1,T2A]{fontenc}
\usepackage[utf8]{inputenc}
\usepackage[american]{babel}
\usepackage{url}
\usepackage{amsfonts}
\usepackage{amssymb, amsthm}
\usepackage{amsmath}
\usepackage{mathtools}
\usepackage{needspace}
\usepackage[pdftex]{graphicx}
\usepackage{hyperref}
\usepackage{datetime}
\usepackage{epigraph}
\usepackage{verbatim}
\usepackage{mathtools}
\usepackage{xcolor}
\usepackage{enumerate}
\usepackage[american]{babel}
\numberwithin{equation}{section}

\newcommand{\Z}{\mathbb{Z}}

\newcommand{\N}{\mathbb{N}}
\newcommand{\R}{\mathbb{R}}

\newcommand{\eps}{\varepsilon}

\renewcommand{\phi}{\varphi}

\newtheorem{Thm}{Theorem}[section]
\newtheorem{theorem}[Thm]{Theorem}
\newtheorem{corollary}[Thm]{Corollary}

\newtheorem{lemma}[Thm]{Lemma}
\newtheorem{proposition}[Thm]{Proposition}
\newtheorem{remark}[Thm]{Remark}

\newtheorem{question}[Thm]{Question}
\newtheorem{definition}{Definition}

\begin{document}
\sloppy

\title[Fekete's lemma in Banach spaces]{Fekete's lemma in Banach spaces}
\author{Aleksei Kulikov}
\address{University of Copenhagen, Department of Mathematical Sciences,
Universitetsparken 5, 2100 Copenhagen, Denmark,
\newline {\tt lyosha.kulikov@mail.ru} 
}
\author{Feng Shao}
\address{Peking University, School of Mathematical Sciences, Beijing, 100871,  China,
\newline {\tt fshao@stu.pku.edu.cn}
}
%\thanks{2010 {\it Mathematics Subject Classification:} Primary 42A16; Secondary 42A20, 42C05}
%\thanks{{\it Key words and phrases:}  prolate spheroidal wave functions, Bargmann--Segal--Fock space, Hermite functions}
\begin{abstract} { For a sequence of vectors $\{v_n\}_{n\in \mathbb{N}}$ in the uniformly convex Banach space $X$ which for all $n, m\in \mathbb{N}$ satisfy $\|v_{n+m}\|\le \|v_n + v_m\|$ we show the existence of the limit $\lim_{n\to \infty} \frac{v_n}{n}$. This extends the classical Fekete's subadditivite lemma to Banach space-valued sequences.
}
\end{abstract}
\maketitle
\section{Introduction}

The classical Fekete's subadditivite lemma \cite{Fekete} says that for a sequence $\{a_n\}_{n\in\N}\subset \R$ such that $a_{n+m} \le a_n + a_m$ for all $n, m\in \N$
there exists a limit $\lim_{n\to \infty} \frac{a_n}{n}\in [-\infty, \infty)$ and moreover that this limit is equal to the infimum $\inf_{n\in \N} \frac{a_n}{n}$ (in this paper, $\N=\Z\cap[1,+\infty)$). This is a beautiful result which has applications in a wide range of mathematical areas such as combinatorics \cite{Lueker,Shur} and functional analysis, in particular giving a very simple proof for the existence of the spectral radius $\lim_{n\to \infty} \|A^n\|^{1/n}$ for a bounded operator $A$, see \cite[Excercise 6.23]{Brezis}.

There are also quite a few generalizations of the Fekete's lemma, such as relaxing the subadditivity condition \cite[Theorem 22 and Theorem 23]{Bruijn_Erdos}, see also \cite[Theorem 1.9.1 and Theorem 1.9.2]{Steele} and \cite{Ruzsa_Furedi} for some recent discussions; or changing the domain of the sequence from $\N$ to $\R_+$ or $\R$ {or even $\R^N$}, see \cite[Section 7.6 and Section 7.13]{Hille_Phillips}. However, in all those cases the range of values was still $\R$. In this paper, instead, we will study what happens if the range is in some Banach space $X$. We begin by recalling the definitions of  convex and uniformly convex Banach spaces.
\begin{definition}
A Banach space $(X,\|\cdot\|)$ is said to be convex if for all $u, v\in X, u\neq v$ with $\|u\| = \|v\| = 1$ we have $\|u+v\| < 2$.
\end{definition}
\begin{definition}
A Banach space $(X,\|\cdot\|)$ is said to be uniformly convex if for any $\eps > 0$ there exists $\delta > 0$ such that for all $u, v\in X$ with $\|u\| = \|v\| = 1, \|u-v\| \ge \eps$ we have $\|u+v\| \le 2-\delta$.
\end{definition}

For finite-dimensional Banach spaces these two notions are equivalent by compactness, but there are infinite-dimensional convex Banach spaces which are not uniformly convex.

The main result of this paper is the following theorem.

\begin{theorem}\label{main}
Let $\{v_n\}_{n\in\N}$ be a sequence of vectors in a uniformly convex Banach space $(X,\|\cdot\|)$. Assume that for all $n, m\in \N$ we have $\|v_{n+m}\| \le \|v_n + v_m\|$. Then there exists a limit $\lim_{n\to \infty} \frac{v_n}{n}$. 
\end{theorem}
This was asked by the second author on MathOverflow \cite{MO} (for the case of the finite-dimensional Hilbert space), and in this paper we present a cleaned up and streamlined proof that was presented on MathOverflow by the first author.

It is worth mentioning that Fekete's subadditive lemma applied to the sequence $\{-a_n\}$ gives that if the sequence satisfies $a_{n+m}\ge a_n + a_m$ then the limit $\lim_{n\to \infty} \frac{a_n}{n} = \sup_{n\in \N} \frac{a_n}{n}$ still exists. However, due to the presence of the norm, in the Theorem \ref{main} we clearly can not change the direction of the inequality -- for example any sequence of vectors $v_n$ such that $\frac{\|v_n\|}{n}$ is constant clearly satisfies $\|v_{n+m}\|\geq\|v_n+v_m\|$ by the triangle inequality but there is no reason for the limit $\lim_{n\to \infty} \frac{v_n}{n}$ to exist.

We also mention an extension of the classical Fekete's subadditive lemma by N. G. de Bruijn and P. Erd\H{o}s \cite[Theorem 22]{Bruijn_Erdos}, stating that if $\{a_n\}_{n\in\N}\subset\R$ is a sequence satisfying $a_{n+m}\leq a_n+a_m$ for all $\frac12 n\leq m\leq 2n$ then $\lim_{n\to\infty}\frac{a_n}{n}=\inf_{n\in\N}\frac{a_n}{n}$. In Section \ref{Sec.BE}, we will show that it only holds in one-dimensional Hilbert spaces, but can fail in the general setup of Theorem \ref{main}.

%taking $$v_n=\mathrm e^{\mathrm i\frac{\pi}{2}\ln n}=H=\R^2=\mathbb C,$$
%then for all $n,m\in\N$ satisfying $\frac12 n\leq m\leq 2n$ we have
%\[\|v_{n+m}\|=1\leq \|v_n+v_m\|=\left|1+\exp\left(\mathrm i\frac\pi2\ln\frac{n}{m}\right)\right|_{\mathbb C},\]
%since $|\ln 1/2|=|\ln 2|<1$, but clearly 

If the Banach space $X$ is not convex, that is if there are distinct vectors $u, v\in X$ with $\|u\| = \|v\| =1, \|u+v\| = 2$ then it is easy to see that Theorem \ref{main} fails for $X$. Indeed, if we take $v_{2n} = 2nv, v_{2n+1} = (2n+1)u$ then this sequence will satisfy even the equality $ \|v_{n+m}\|=\|v_n + v_m\|$, but the limit $\lim_{n\to\infty} \frac{v_n}{n}$ clearly does not exist. In particular, since finite-dimensional space is convex if and only if it is uniformly convex, we get the following simple corollary.
\begin{corollary}
Let $(X,\|\cdot\|)$ be a finite-dimensional Banach space. Then $X$ is convex if and only if for all sequences $\{v_n\}_{n\in\N}$ of vectors in $X$ such that $\|v_{n+m}\|\le \|v_n+v_m\|$ holds for all $n, m\in \N$, the limit $\lim_{n\to \infty}\frac{v_n}{n}$ exists.
\end{corollary}

If the Banach space is infinite-dimensional, the situation becomes less clear. In Section 4 we will present an example of a convex Banach space for which Theorem \ref{main} fails, as well as an example of a non-uniformly convex Banach space for which Theorem \ref{main} holds. So, for a Banach space, the assumption that the Fekete's lemma holds is strictly between convexity and uniform convexity.

Although we do not know a necessary and sufficient condition on the Banach space $X$ for the Fekete's lemma to hold in $X$, in Section 5 we present a criterion on the convex Banach space $X$ together with a subadditive sequence $\{v_n\}_{n\in\N}\subset X\backslash\{0\}$ for the existence of the limit $\lim_{n\to\infty} \frac{v_n}{n}$, which is similar to Theorem \ref{main} in that it requires a uniform convexity assumption: the limit $\lim_{n\to\infty} \frac{v_n}{n}$ exists if and only if either $\lim_{n\to\infty} \frac{\|v_n\|}{n} = 0$ or the sequence of vectors $\{\frac{v_n}{\|v_n\|}\}_{n\in\N}$ is a uniformly convex subset of $X$.

\section{Proof of Theorem \ref{main}}
We begin by noting that from the usual Fekete's lemma and the triangle inequality we get that $L = \lim_{n\to \infty} \frac{\|v_n\|}{n} = \inf_{n\in \N} \frac{\|v_n\|}{n}< \infty$ exists. Clearly, $L\ge 0$ since it is a limit of non-negative quantities. If $L = 0$ then $\frac{v_n}{n}$ tends to $0$ and there is nothing to prove. So, without loss of generality we can assume that $0 < L < \infty$. By replacing $v_n$ with $\frac{v_n}{L}$ we can also assume that $L = 1$. At the heart of our proof is the following purely geometric lemma.
\begin{lemma}\label{lem}
Let $X$ be a uniformly convex Banach space and let $u, v\in X$ be non-zero vectors such that $\left\|\frac{u}{\|u\|} - \frac{v}{\|v\|}\right\| \ge \eps$ and such that $\|v\| \le 2 \|u\|$. There exists constant $\gamma = \gamma(\eps) < 1$ such that 
\begin{equation}\label{eq}
\|u + v\| \le \|u\| + \gamma \|v\|.
\end{equation}
\end{lemma}
This lemma was initially proved by us only for the case when $X$ is a Hilbert space by a direct computation with the law of cosines. We thank Fedor Petrov for suggesting us a simple proof of this lemma for all uniformly convex Banach spaces, and for allowing us to include it in our text.
\begin{proof}
Put $u_1 = \frac{\|v\|}{2\|u\|}u$, $u_2 = u-u_1 = \left(1-\frac{\|v\|}{2\|u\|}\right)u$. We have $\|u_1\| = \frac{\|v\|}{2}$, $\|u_2\| = \|u\|-\|u_1\|$ since $\|v\| \le 2\|u\|$. We have $u + v = u_1 + u_2 + \frac{v}{2} + \frac{v}{2} = u_2 + \frac{v}{2} + (u_1 + \frac{v}{2})$. For the sum in brackets we have $\|u_1 + \frac{v}{2}\| = \frac{\|v\|}{2} \left\|\frac{u}{\|u\|}+\frac{v}{\|v\|}\right\|$. By the uniform convexity and our assumption on $u$ and $v$ this is at most $\frac{\|v\|}{2}(2-\delta)$ for some $\delta > 0$. Using the triangle inequality we get
$$\|u+v\| \le \|u_2\| + \left\|\frac{v}{2}\right\| + \left\|u_1 + \frac{v}{2}\right\| \le \|u\|-{\frac{\|v\|}2} + \frac{\|v\|}{2} + \frac{\|v\|}{2}(2-\delta) = \|u\| + \|v\|\left(1 - \frac{\delta}{2}\right),$$
so $\gamma = 1-\frac{\delta}{2}$ works.
\end{proof}
\begin{remark}
The constant $2$ in the inequality $\|v\| \le 2\|u\|$ is not important, for any constant in its place we would get a similar bound ($\gamma$ would of course depend on this constant). For our proof of Theorem \ref{main} any positive constant bigger than $1$ would work.
\end{remark}

\begin{remark}
By taking $u$ and $v$ with $\|u\| = \|v\| = 1$ it is easy to see that Lemma \ref{lem} implies that $X$ is uniformly convex. Thus, $X$ is uniformly convex if and only if Lemma \ref{lem} holds for $X$.
\end{remark}

Our goal is to prove that the sequence $w_n =\frac{v_n}{n}$ is Cauchy. Since $\|w_n\|\to 1$,  this is the same as saying that the sequence $\frac{w_n}{\|w_n\|} = \frac{v_n}{\|v_n\|}$ is Cauchy. We will do this in two steps. First, we will show that $\left\|\frac{v_n}{\|v_n\|}-\frac{v_m}{\|v_m\|}\right\|$ is small if the ratio between $n$ and $m$ is not too big.
\begin{proposition}\label{prop} For all $\eps > 0$ there exists $N$ such that for all $N \le n \le m \le 4n$ we have $\left\|\frac{v_n}{\|v_n\|} - \frac{v_m}{\|v_m\|}\right\| < \eps$.
\end{proposition}
\begin{proof}
Since $\frac{\|v_k\|}{k}$ tends to $1$, for any $\delta > 0$ there exists $N_\delta$ such that $\|v_k\| \le (1+\delta)k$ for $k > N_{\delta}$. On the other hand, from the Fekete's lemma we know that $\|v_k\| \ge k$ for all $k\in \N$. 

First, we choose $N > N_1$ so that $\|v_n\| \le 2n \le 2m \le 2\|v_m\|$. Hence, we can apply Lemma \ref{lem} to $v_m$ and $v_n$. If $\|\frac{v_n}{\|v_n\|}-\frac{v_m}{\|v_m\|}\| \ge \eps$ then $$\|v_m + v_n\| \le \|v_m\| + \gamma \|v_n\|.$$
By choosing $N > N_\delta$ this is at most $(1+\delta)m + \gamma (1+\delta)n = (1+\delta)(m+\gamma n)$. On the other hand, this is at least $\|v_{n+m}\|$ which in itself is at least $n+m$. Combining this we get
$$(1+\delta)(m+\gamma n) \ge n+m.$$
Rearranging, we get $\delta m \ge n(1-(1+\delta)\gamma)$. On the other hand, since $m \le 4n$ we have $4\delta n \ge (1-(1+\delta)\gamma)n$. Dividing by $n$ we get $4\delta \ge (1-(1+\delta)\gamma)$, but this is false for small enough $\delta$ since as $\delta \to 0$ the left-hand side tends to $0$ while the right-hand side tends to $1-\gamma > 0$. So, for small enough $\delta$ if $N > N_\delta$ then $\left\|\frac{v_n}{\|v_n\|}-\frac{v_m}{\|v_m\|}\right\| < \eps$.
\end{proof}

It remains to show that if the ratio between $n$ and $m$ is big the distance between $\frac{v_n}{\|v_n\|}$ and $\frac{v_m}{\|v_m\|}$ must nevertheless go to $0$. For this, one application of the inequality $\|v_{n+m}\|\le \|v_n + v_m\|$ is not enough, we have to use it multiple times. So, assume that $m > 4n > n > N$ for a big number $N$ and that $\|\frac{v_n}{\|v_n\|}-\frac{v_m}{\|v_m\|}\| \ge 2\eps$ for some fixed $\eps > 0$. Then we look for a contradiction.

Applying Lemma \ref{lem} to $v_m$ and $v_n$ we get that $\|v_{n+m}\|\le \|v_m\|+\gamma \|v_n\|$ (note that here, just as in the proof of Proposition \ref{prop}, we can assume that $\|v_n\| \le 2 \|v_m\|$ if $N$ is big enough). On the other hand, $n + m \le m + m = 2m$, therefore we can apply Proposition \ref{prop} to the vectors $v_m$ and $v_{n+m}$ and get $\left\|\frac{v_m}{\|v_m\|}-\frac{v_{n+m}}{\|v_{n+m}\|}\right\| < \eps$ if $N$ is big enough. Hence, by the triangle inequality we have $\left\|\frac{v_n}{\|v_n\|}-\frac{v_{n+m}}{\|v_{n+m}\|}\right\| \ge \eps$. Thus, we can apply Lemma \ref{lem} to $v_n$ and $v_{n+m}$, then $$\|v_{2n+m}\| \le \|v_{n+m}+v_n\| \le \|v_{n+m}\| + \gamma \|v_n\| \le \|v_m\| + 2\gamma \|v_n\|.$$

We will continue doing this up to the vector $v_{kn+m}$, where $k$ is the smallest integer such that $kn \ge m$. More precisely, we will show by induction that 
\begin{equation}\label{intermediate}
\|v_{rn+m}\|\leq\|v_m\|+r\gamma\|v_n\|
\end{equation}
holds for all $r\le k$. For $r = 1, 2$ we already obtained it. Assume that it holds for some $r < k$.  Then $N\le m\le rn + m \le m+m = 2m$, hence we can apply Proposition \ref{prop} to $v_m$ and $v_{{rn}+m}$ which tells us that $\left\|\frac{v_m}{\|v_m\|}-\frac{v_{rn+m}}{\|v_{rn+m}\|}\right\| < \eps$. Since $\left\|\frac{v_n}{\|v_n\|}-\frac{v_m}{\|v_m\|}\right\| \ge 2\eps$, by the triangle inequality we get $\left\|\frac{v_n}{\|v_n\|}-\frac{v_{rn+m}}{\|v_{rn+m}\|}\right\| \ge \eps$. Note that $\|v_n\|\leq 2n\leq 2(rn+m)\leq 2\|v_{rn+m}\|$, therefore Lemma \ref{lem} implies that
\[\|v_{(r+1)n+m}\|\leq\|v_n+v_{rn+m}\|\leq \|v_{rn+m}\|+\gamma\|v_n\|\leq \|v_m\|+(r+1)\gamma\|v_n\|,\]
so \eqref{intermediate} holds for $r+1$.
Taking $r=k$ we get
$$\|v_{kn+m}\| \le \|v_m\| + k\gamma \|v_n\|.$$

If $N$ is big enough we can assume that $\|v_m\| \le (1+\delta)m$, $\|v_n\| \le (1+\delta)n$. On the other hand, $\|v_{kn+m}\| \ge kn + m$. Thus, 
$$kn+m \le (1+\delta)(m + kn\gamma).$$
Rearranging this we get $kn(1-(1+\delta)\gamma) \le \delta m$. On the other hand, $kn \ge m$, so we deduce that $(1-(1+\delta)\gamma) \le \delta$. And now, just like in the proof of Proposition \ref{prop}, this is false for small enough $\delta$. Hence, for $N > N_{\delta}$ we have $\left\|\frac{v_m}{\|v_m\|}-\frac{v_n}{\|v_n\|}\right\|<2\eps$. This completes the proof of Theorem \ref{main}.

\section{Extension of the de Bruijn--Erd\H{o}s subadditive lemma}\label{Sec.BE}
In this section, we will try to extend the de Bruijn--Erd\H{o}s subadditive lemma to the Banach space-valued sequences. We start with one-dimensional spaces.

\begin{proposition}\label{Prop.1-D}
    Let $\{a_n\}_{n\in\N}\subset \R$ be a sequence satisfying
    \begin{equation}\label{Eq.subadditive}
        |a_{n+m}|\leq |a_n+a_m|,\quad\text{~for all~}\ \frac12 n\leq m\leq 2n.
    \end{equation}
    Then the limit $\lim_{n\to\infty}\frac{a_n}{n}\in\R$ exists.
\end{proposition}
\begin{proof}
    By \eqref{Eq.subadditive}, the triangle inequality and \cite[Theorem 22]{Bruijn_Erdos}, there exists $L\in[0,+\infty)$ such that
    \[\lim_{n\to\infty}\frac{|a_n|}{n}=L=\inf_{n\in\N}\frac{|a_n|}{n}.\]
    If $L=0$, then $\lim_{n\to\infty}\frac{a_n}{n}=0$. Hence we assume $L>0$. For any $\varepsilon\in(0, L/4)$, there exists $N>2$ such that
    \[\left|\frac{|a_n|}{n}-L\right|<\varepsilon\quad\text{~for all~}\ n\geq N.\]
    Hence,
    \begin{equation}\label{Eq.a_n}
        n(L-\varepsilon)<|a_n|<n(L+\varepsilon)\quad\text{~for all~}\ n\geq N.
    \end{equation}
    
    We claim that if $\varepsilon$ is small then for all $n\geq N$ we have $a_na_{n+1}>0$. Indeed, suppose not, then there exists some $n\geq N$ such that $a_na_{n+1}<0$ (by \eqref{Eq.a_n} we know that $a_n\neq 0$ for $n\geq N$), thus by \eqref{Eq.subadditive} and \eqref{Eq.a_n} we have
    \[(2n+1)(L-\varepsilon)\leq |a_{2n+1}|\leq |a_n+a_{n+1}|\leq L+(2n+1)\varepsilon.\]
    In particular, we have $(2n+1)(L-2\varepsilon)\le L$ for some $n\geq N$. Taking $\varepsilon = \frac{L}{8}$ and $N > 2$ gives a contradiction. This proves our claim.

    Hence, $a_n$ has the same sign for all sufficiently large $n$. Therefore, the limit $\lim_{n\to\infty}\frac{a_n}{n}$ exists and equals to $L$ or $-L$.
\end{proof}

Now we will show that even if in Proposition \ref{Prop.1-D} we replace $\R$ by $\R^2$ with the Euclidean norm, the conclusion fails. Consider the sequence $r_n = n +\frac{n}{\ln(n+1)^{1/2}}$, it clearly satisfies $\lim_{n\to\infty}\frac{r_n}{n}=1>0$. We claim that
\begin{equation}\label{Eq.r_n}
    r_{n+m}^2\leq r_n^2+r_m^2+2\left(1-\frac {1}{100\ln( n+1)^{3/2}}\right)r_nr_m\quad\text{~for all~}\ n\leq m\leq 2n,\ n\in\N.
\end{equation}
This can be rewritten as $$(r_n+r_m-r_{n+m})(r_n + r_m +r_{n+m}) \ge  \frac{2r_n r_m}{100\ln(n+1)^{3/2}}.$$
We have $$r_n + r_m - r_{n+m} \ge \frac{n}{\ln(n+1)^{1/2}}-\frac{n}{\ln(n+m+1)^{1/2}} \ge n\left(\frac{1}{\ln(n+1)^{1/2}} - \frac{1}{\ln(2n+1)^{1/2}}\right),$$
and $r_n + r_m + r_{n+m} \ge r_m$, while 
$$\frac{2r_n r_m}{100\ln(n+1)^{3/2}} \le \frac{nr_m}{20 \ln(n+1)^{3/2}}.$$
Thus, it remains to show that
$$\frac{1}{\ln(n+1)^{1/2}} - \frac{1}{\ln(2n+1)^{1/2}} \ge \frac{1}{{20}\ln(n+1)^{3/2}}.$$
The left-hand side is equal to
$$\frac{\ln(2-\frac{1}{n+1})}{\ln(n+1)^{1/2}\ln(2n+1)^{1/2}(\ln(n+1)^{1/2}+\ln(2n+1)^{1/2})},$$
which is at least $\frac{\ln(3/2)}{2\ln(2n+1)^{3/2}}$. It remains to note that $$\ln(2n+1) \le \ln((n+1)^2)= 2\ln(n+1),$$ so the inequality holds even with the constant $\frac{2^{5/2}}{\ln(3/2)} < 14$ in place of $20$.

Now, we take $\theta_n=\frac{(\ln n)^{1/4}}{100}$ for all $n\in\N$. We claim that
\begin{equation}\label{Eq.theta_2n-n}
0\le \theta_{2n} - \theta_{n} = \frac{\ln (2n)^{1/4} -\ln(n)^{1/4}}{50}\le \frac{\ln(n+1)^{-3/4}}{50},
\end{equation}
where the left inequality is obvious and the right inequality is true for $n = 1$, while for $n > 1$ we have $\ln(2n)^{1/4} \le (\ln(n) + 1)^{1/4} \le \ln(n)^{1/4} + \frac{1}{4\ln(n)^{3/4}}$ since the function $x\to x^{1/4}$ is concave, so
$\ln(2n)^{1/4} - \ln(n)^{1/4} \le \frac{1}{4\ln(n)^{3/4}},$ and $\frac{\ln(n+1)}{\ln(n)} \le \frac{\ln(n^2)}{\ln(n)} = 2$, thus
$$\theta_{2n}-\theta_n \le \frac{2^{3/4}}{200}\ln(n+1)^{-3/4} \le \frac{\ln(n+1)^{-3/4}}{50}.$$

Consider the sequence of vectors $v_n = (r_n\cos(\theta_n), r_n\sin(\theta_n))\in \R^2$. The length of $\frac{v_n}{n}$ converges to $1$, but since $\theta_n$ tends to $+\infty$ while $\theta_{n+1}-\theta_n$ tends to $0$, the limit of the vector $\frac{v_n}{n}$ does not exist. It remains to show that for $\frac{1}{2}n\le m\le 2n$ we have $\|v_n+v_m||\ge \|v_{n+m}||$. Without loss of generality we can assume that $n \le m$, otherwise we can simply swap $n$ and $m$. We have
\begin{align*}
    \|v_n+v_m\|^2=r_n^2+r_m^2+2r_nr_m\cos(\theta_{m}-\theta_n);
\end{align*}
it follows from the monotonicity of $\{\theta_n\}$ and \eqref{Eq.theta_2n-n} that
\[\cos(\theta_{m}-\theta_n)\geq 1-\frac12(\theta_{m}-\theta_n)^2\geq 1-\frac12(\theta_{2n}-\theta_n)^2\geq 1-\frac{1}{100}(\ln n+1)^{-3/2},\]
hence \eqref{Eq.r_n} implies that
\[\|v_n+v_m\|^2\geq r_n^2+r_m^2+2\left(1-\frac {1}{100(\ln n+1)^{3/2}}\right)r_nr_m\geq r_{n+m}^2=\|v_{n+m}\|^2,\]
as required.

Moreover, by taking $r_n = n + ns_n$ with $s_n = (\ln n)^{-\delta}$ and $\theta_n = (\ln n)^{\delta}$ for big enough $n$, where $\delta$ is a small positive number, 
by the same method we can {check} the inequality $\|v_{n+m}\| \le \|v_n+v_m\|$ for an even wider range of pairs $(n, m)$, specifically for $n \le m \le n \exp((\ln n)^{1-\eps})$ where $\eps = \eps(\delta)$ tends to $0$ as $\delta$ tends to $0$. We do not know whether there exist functions $f(n)$ such that the Banach space-valued version of the Fekete's subadditivity lemma is true if we only assume it for $n \le m \le f(n)$.
\begin{question}Let $X$ be a uniformly convex Banach space. Does there exist a function $f:\N\to\N$ such that for all sequences $v_n\in X$ with $\|v_{n+m}\| \le \|v_n + v_m\|$ for $n \le m \le f(n)$ the limit $\lim \frac{v_n}{n}$ exists?
\end{question}
By the Proposition \ref{Prop.1-D} if $X$ is one-dimensional then $f(n) = 2n$ is enough, but already for the two-dimensional Hilbert spaces $X$ we need a superlinear function $f$.
\section{Banach space case}
As we said in the introduction, if the Banach space is not convex then  Fekete's lemma fails in it. Here, we will construct a convex Banach space in which Fekete's lemma is still false.

Consider the Banach space $X$ of sequences $(a_1, a_2, \ldots)$ for which the norm
$$\|(a_1, a_2, \ldots)\| = \sum_{n = 1}^\infty |a_n| + \sqrt{\sum_{n = 1}^\infty \frac{|a_n|^2}{16^n}}$$
is finite. This is clearly a norm (in fact, $X$ is nothing but $\ell^1(\N)$, just with a slightly different norm), and due to the presence of the quadratic term the space $X$ with this norm is strictly convex. However, if we take $v_n = ne_n$, where $e_n$ is the vector $(0, \ldots 0, 1, 0, \ldots)$, where $1$ is only on $n$'th place, then this sequence satisfies our assumption:
$$\|v_n+v_m\| = n + m + \sqrt{\frac{n^2}{16^n}+\frac{m^2}{16^m}} \ge n + m + \frac{n}{4^n}\ge n + m + \frac{n+m}{4^{n+m}} = \|v_{n+m}\|,$$
where the second inequality is true because the function $\frac{x}{4^x}$ is decreasing on $[1, \infty)$. On the other hand, the limit $\lim_{n\to \infty} \frac{v_n}{n}$ clearly does not exist.

This example is a slight perturbation of a non-convex Banach space $\ell^1(\N)$. Yet, since it is convex it shows that convexity is not a sufficient condition for the Fekete's lemma to hold. Now, we will show an example which demonstrates that uniform convexity is not necessary for it to hold either, hence the condition that the Fekete's lemma holds is strictly between convexity and uniform convexity. This example was communicated to us by Dongyi Wei and we thank him for allowing us to include it in our text.

Consider the Banach space X of sequences $x=(x_1, x_2, \ldots)$ for which the norm
$$\|x\| = \left(\sum_{k=1}^\infty \left(|x_{2k-1}|^{k+1}+|x_{2k}|^{k+1}\right)^{2/(k+1)}\right)^{1/2}= \|\left(\|(x_{2k-1}, x_{2k})\|_{\ell^{k+1}}\right)\|_{\ell^2(k\in\N)}.$$
As a set $X$ is equal to $\ell^2(\N)$ but unlike $\ell^2(\N)$ it is not uniformly convex: for the sequence of vectors $v_n = \frac{e_{2n-1}+e_{2n}}{2^{1/(n+1)}}, w_n = \frac{e_{2n-1}-e_{2n}}{2^{1/(n+1)}}$, where $e_k$ is a vector with $1$ at the $k$'th position and zero everywhere else, we have $\|v_n\| = \|w_n\| = 1$, $\|v_n-w_n\|{=2^{\frac{n}{n+1}} \ge1}$ and $\|v_n+w_n\| = 2^{1-1/(n+1)}\to 2$ as $n\to \infty$. Yet, as we will now show, the Fekete's lemma holds for $X$.

Let $v_n$ be a sequence of vectors in $X$ such that $\|v_{n+m}\|\le \|v_n + v_m\|$ for all $n, m\in \N$. Consider the vectors $w_n\in \ell^2(\N)$ defined by
$$w_{n, k} = \left(|v_{n, 2k-1}|^{k+1}+|v_{n, 2k}|^{k+1}\right)^{1/(k+1)}.$$

We have $\|w_n\|_{\ell^2} = \|v_n\|_X$ and $\|w_n+w_m\|_{\ell^2} \ge \|w_{n+m}\|_{\ell^2}$, where the first equality is obvious and the second inequality holds because of the subadditivity assumption on $v_n$ and the triangle inequality in $\ell^{k+1}$ for each $k$. By Theorem \ref{main} and the uniform convexity of $\ell^2(\N)$ we get the existence of the limit $W = \lim_{n\to \infty} \frac{w_n}{n}$.

Our next step is to show that for each $k$ the limit of the vector $\frac{u_{n, k}}{n}=(\frac{v_{n, 2k-1}}{n}, \frac{v_{n, 2k}}{n}) $ exists. If $W_k = 0$ then $\frac{\|u_{n, k}\|_{\ell^{k+1}}}{n}$ tends to $0$, therefore $\frac{u_{n, k}}{n}$ tends to $0$. So, from now on we assume that $W_k > 0$. The idea now is to write the inequality $\|v_{n+m}\|\le \|v_n + v_m\|$ and to ignore all but $2k-1$ and $2k$'th coordinates by the triangle inequality. We get
$$\|v_{n+m}\| \le \left(\left(|v_{n, 2k-1}+v_{m, 2k-1}|^{k+1}+|v_{n, 2k} + v_{m, 2k}|^{k+1}\right)^{2/(k+1)}+\sum_{l\neq k} \left(w_{n, l}+w_{m, l}\right)^2\right)^{1/2}.$$

Now, we can essentially repeat the proof of Theorem \ref{main}. If $n$ and $m$ are close, meaning $n \le m \le 4n$ then the vectors $\frac{u_{n, k}}{n}$ and $\frac{u_{m, k}}{m}$ must be close otherwise we will lose a positive portion of the norm which is not allowed, and then if $4n < m$ we apply the above inequality many times to again get that if the vectors are not close then eventually we will lose a positive portion of the norm. Both of these parts crucially rely on the fact that, for fixed $k$, $\ell^{k+1}$ is a uniformly convex Banach space. We leave the details of the computations to the interested reader.

It remains to show that the limit of $\frac{v_n}{n}$ exists as a vector. For each $k$ we showed the existence of the limits $\frac{v_{n, 2k-1}}{n}$ and $\frac{v_{n, 2k}}{n}$, so let us collect all of these limits into a single vector $V$. Our goal is to show that $\lim_{n\to\infty} \frac{v_n}{n} = V$. Computing the norm of $\frac{v_n}{n} - V$, for each of the first $K$ coordinates we know that their contributions tend to $0$ for any $K$, in particular can be made less than arbitrary $\delta > 0$. On the other hand, since we already have a limit vector $w$, we can for any $\eps > 0$ finds $K$ such that $\sum_{k > K} W_k^2 < \eps^2$. Since $\frac{w_n}{n}\to W$ we get that for big enough $n$ we have $\sum_{k > K} w_{n, k}^2 < 2{n^2}\eps^2$.  Then by Fatou's lemma  for big enough $n$ we get $\|\frac{v_n}{n}-V\| \le (K\delta^2 + 4\eps^2)^{1/2}$ which can be made less than $3\eps$ if $\delta$ is small enough depending on $K$ and $\eps$. Since $\eps$ is arbitrary we get that $\lim_{n\to \infty}\frac{v_n}{n} = V$ as required.

The key idea of this example is that we took a sequence of uniformly convex Banach spaces which is not {\it uniformly} uniformly convex, and put them into an $\ell^2$-series. Then, the proof consisted of essentially two applications of the main result, first on the encompassing $\ell^2$ space, and then on each uniformly convex Banach space individually, where we did not care anymore if the triangle inequality is uniform across our Banach spaces. Of course, this procedure can be iterated, so we can produce even more examples by taking nested uniformly convex Banach spaces. Note also that the fact that $\ell^{k+1}$ that we considered were finite-dimensional played no role, the proof would have worked even if they were infinite-dimensional, we made them two-dimensional only for the typographical reasons.

We end this section by briefly describing an example showing that the fact that $X$ is Banach is also necessary, that is if $X$ is uniformly convex but not complete the Fekete's lemma may not hold in $X$. Take $X$ to be the set of sequences $(a_1, a_2, \ldots)$ which are eventually $0$ and equip it with an $\ell^2$-norm. Define the sequence $v_n\in X$ by $$v_{n, m} = \begin{cases}\frac{c_n}{2^{2^m}},\, m\le n,\\
0,\quad\, m > n,\end{cases}$$
where $c_n = n(1 + \frac{10}{\log\log(10n)})$. We leave the verification of the inequality $\|v_{n+m}\|\le\|v_n+v_m\|$ for all $n, m\in \N$ to the interested reader, so that the sequence $\{v_n\}_{n\in\N}$ satisfies the assumption of the Fekete's lemma. On the other hand, the limit $V=\lim_{n\to \infty}\frac{v_n}{n}$ must satisfy $V_m = \frac{1}{2^{2^m}}$, so $V\notin X$.

\section{A necessary and sufficient condition}
The previous discussions show that, given a Banach space $X$, we have not found an efficient way to determine whether Fekete's lemma holds for $X$ or not. Convexity is necessary but not sufficient, and uniform convexity is sufficient but not necessary. In this section, we give a criterion to determine, in a general convex Banach space, given a sequence satisfying the hypothesis of Fekete's lemma, whether the conclusion of Fekete's lemma holds or not. 

We start with the following definition.
\begin{definition}
    Let $(X,\|\cdot\|)$ be a normed space and let $S\subset X$. The subset $S$ is called uniformly convex if for all $\varepsilon>0$ there exists $\delta\in(0,1)$ such that for all $u,v\in S$ satisfying $\|u\|=\|v\|=1$ and $\|u-v\|\geq \varepsilon$, there holds $\|u+v\|\leq 2-\delta$.
\end{definition}
\begin{remark}
    If
$S\cap \{x\in X: \|x\|=1\}=\varnothing,$
then $S$ is uniformly convex by default.
\end{remark}

The following result follows directly from the proof of our main theorem.
\begin{corollary}
    Let $(X,\|\cdot\|)$ be a Banach space and let $\{v_n\}_{n\in\mathbb N}\subset X$ be a sequence such that 
	\begin{equation}\label{Eq.Fekete}
		\|v_{n+m}\|\leq \|v_n+v_m\| \quad \text{for all }n, m\in\mathbb N.
	\end{equation} 
	Assume that either
	\begin{itemize}
		\item there exists $i\in\mathbb N$ such that $v_i=0$, or 
		\item $v_n\neq 0$ for all $n\in\mathbb N$ and $\{\frac{v_n}{\|v_n\|}\}_{n\in\mathbb N}\subset X$ is a uniformly convex subset.
	\end{itemize} 
	Then the limit $\lim_{n\to\infty}\frac{v_n}{n}$ exists.
\end{corollary}
\begin{proof}
	Observe that if $v_i=0$ for some $i\in\mathbb N$, then by \eqref{Eq.Fekete} $v_{ki}=0$ for all $k\in\mathbb N$. Since the limit $\lim_{n\to\infty}\frac{\|v_n\|}{n}$ exists by the classical Fekete's lemma, we know that $\lim_{n\to\infty}\frac{\|v_n\|}{n}=0$, thus $\lim_{n\to\infty}\frac{v_n}{n}=0$ in $X$.
	
    If $v_n\neq 0$ for all $n\in\mathbb N$, then we can show that $\{\frac{v_n}{\|v_n\|}\}_{n\in\mathbb N}$ is Cauchy, using the same proof as our Theorem \ref{main}.
\end{proof}

The opposite direction is based on the following proposition.
 \begin{proposition}
        Let $(X,\|\cdot\|)$ be a convex normed space and let $S\subset X$ be a subset in which any sequence has a convergent subsequence in $X$. Then $S$ is uniformly convex.
    \end{proposition}
    \begin{proof}
        Let $\varepsilon>0$. Assume on the contrary that there exist two sequences $\{u_n\}_{n\in\mathbb N}\subset S$ and $\{v_n\}_{n\in\mathbb N}\subset S$ such that $\|u_n\|=1=\|v_n\|$, $\|u_n-v_n\|\geq \varepsilon$ for all $n\in\mathbb N$ and $\|u_n+v_n\|\to 2$ as $n\to\infty$. By passing to the subsequences, we can assume without loss of generality that $u_n\to u\in X$ and $v_n\to v\in X$ as $n\to\infty$. Hence $\|u\|=1=\|v\|$, $\|u-v\|\geq \varepsilon$ and $\|u+v\|=2$. This contradicts the convexity of $X$.
    \end{proof}

Immediately from this, we get the following corollary.
\begin{corollary}
    Let $(X,\|\cdot\|)$ be a convex normed space. Assume that $\{v_n\}_{n\in\mathbb N}\subset X$ is a sequence satisfying
	\[\lim_{n\to\infty}\frac{v_n}{n}\in X\setminus\{0\},\]
	then there exists $N_0$ such that $\{\frac{v_n}{\|v_n\|}\}_{n>N_0}$ is a uniformly convex subset. Moreover, if $v_n\neq 0$ for all $n\in\mathbb N$, then $\{\frac{v_n}{\|v_n\|}\}_{n\in\mathbb N}$ is a uniformly convex subset.
\end{corollary}

We conclude this section with the following criterion.
\begin{theorem}
    Let $(X, \|\cdot\|)$ be a convex Banach space and let $\{v_n\}_{n\in\mathbb N}\subset X$ be a sequence such that $\|v_{n+m}\|\leq \|v_n+v_m\|$ for all $n,m\in\N$ and $\lim_{n\to\infty}\frac{\|v_n\|}{n}>0$. Then the limit $\lim_{n\to\infty}\frac{v_n}{n}$ exists in $X$ if and only if $\{\frac{v_n}{\|v_n\|}\}_{n\in\N}$ is a uniformly convex subset.
\end{theorem}

\subsection*{Acknowledgments} We want to thank Fedor Petrov for helpful comments, and for suggesting us the proof of Lemma \ref{lem} in full generality.  We also would like to thank Dongyi Wei for showing us an example of a non-uniformly convex Banach space in which the Fekete's lemma holds. Aleksei Kulikov was supported by BSF Grant 2020019, ISF Grant 1288/21, by The Raymond and Beverly Sackler Post-Doctoral Scholarship and by the VILLUM Centre of Excellence for the Mathematics of Quantum Theory (QMATH) with Grant No.10059.


\begin{thebibliography}{9999}



\bibitem{Brezis} H. Brezis, {\it Functional analysis, Sobolev spaces and partial differential equations}. Universitext, Springer, New York, 2011. xiv+599 pp.

\bibitem{Bruijn_Erdos} N. G. de Bruijn and P. Erd\"os, Some linear and some quadratic recursion formulas. II. {\it Indag. Math.}, {\bf 14} (1952), 152–163.

\bibitem{Fekete} M. Fekete, \"Uber die Verteilung der Wurzeln bei gewissen algebraischen Gleichungen mit ganzzahligen Koeffizienten. {\it Math. Z.}, {\bf 17} (1923), 228–249.

\bibitem{Hille_Phillips} E. Hille and R. S. Phillips, {\it Functional Analysis And Semi-Groups, rev. ed}. Amer. Math. Soc. Colloq. Publ., {\bf 31}. American Mathematical Society, Providence, RI, 1957. xii+808 pp.

%\bibitem{Kulikov} A. Kulikov (\url{https://mathoverflow.net/users/104330/aleksei-kulikov}), High dimensional Fekete's subadditive lemma: does $|\vec x_{n+m}|\leq |\vec x_n+\vec x_m|$ imply the convergence of $\{\vec x_n/n\}$?, URL (version: 2023-12-05): \url{https://mathoverflow.net/q/459785}.

\bibitem{Lueker} G. S. Lueker, Improved bounds on the average length of longest common subsequences. {\it J. ACM}, {\bf 56} (2009), Art. 17, 38 pp.

\bibitem{Ruzsa_Furedi} I. Z. Ruzsa and Z. F\"uredi, Nearly subadditive sequences. {\it Acta Math. Hungar.}, {\bf 161} (2020), 401–411.

\bibitem{MO} 
F. Shao (\url{https://mathoverflow.net/users/141451/feng}) and A. Kulikov (\url{https://mathoverflow.net/users/104330/aleksei-kulikov}), High dimensional Fekete's subadditive lemma: does $|\vec x_{n+m}|\leq |\vec x_n+\vec x_m|$ imply the convergence of $\{\vec x_n/n\}$?, URL (version: 2023-12-05): \url{https://mathoverflow.net/q/459773}.

\bibitem{Shur} A. M. Shur, Growth properties of power-free languages. {\it Comput. Sci. Rev.}, {\bf 6} (2012), 187–208. 
 
\bibitem{Steele} J. M. Steele, {\it Probability theory and combinatorial optimization}. CBMS-NSF Regional Conf. Ser. in Appl. Math., {\bf 69}, Society for Industrial and Applied Mathematics (SIAM), Philadelphia, PA, 1997. viii+159 pp.
 

\end{thebibliography}
\end{document}